\documentclass{article}
\usepackage{amsmath}
\usepackage{amssymb}
\usepackage{amsbsy}
\usepackage{mathrsfs}
\usepackage[all]{xypic}

\newcommand{\ra}{\rightarrow}

\newcommand{\bb}[1]{\mathbb{#1}}

\newcommand{\Z}{\bb{Z}}
\newcommand{\Q}{\bb{Q}}
\newcommand{\R}{\bb{R}}
\newcommand{\C}{\bb{C}}
\newcommand{\Zp}{\Z_p}

\newcommand{\Qp}{\Q_p}

\newcommand{\pp}{\mathfrak{P}}

\newcommand{\ten}{\otimes}
\newcommand{\teno}[1]{\ten_{#1}}

\newcommand{\gal}[1]{\mathrm{Gal}(#1)}
\newcommand{\of}{\circ}

\newcommand{\br}[1]{\bar{#1}} 
\newcommand{\fd}[2]{{[#1:#2]}}
\newcommand{\st}{^\times}
\newcommand{\ie}{i.e. }

\newcommand{\can}{\simeq}
\newcommand{\Hom}{\mathrm{Hom}}

\newcommand{\End}{\mathrm{End}}

\newcommand{\Det}{\textstyle{\det}} 

\newcommand{\spec}{\mathrm{Spec}\spc}

\newcommand{\oh}{\mathcal{O}}
\newcommand{\class}{\mathrm{Cl}}
\newcommand{\gen}[1]{\langle #1 \rangle}

\newcommand{\abqp}{\gal{\br{\Q}_p/{\Qp}}}

\newcommand{\spc}{\mbox{ }}
\newcommand{\Mod}{\spc\mathrm{mod}\spc}

\newcommand{\con}{\subseteq}

\newcommand{\ann}{\mathrm{ann}}

\newcommand{\iJ}[1]{\mathcal{J}(#1)}

\newcommand{\E}{\mathcal{E}}
\newcommand{\eps}{\epsilon}

\newcommand{\sat}{\spc | \spc}

\newcommand{\stick}{\theta}

\newcommand{\absnm}{\mathbf{N}}

\newcommand{\rk}{\mathrm{rk}}

\newcommand{\fitt}{\mathrm{Fitt}}
\newcommand{\blf}[2]{\gen{#1,#2}}

\newcommand{\chr}{\mathrm{char}}

\newcommand{\Cp}{\mathbb{C}_p}

\newcommand{\chgrp}[1]{\widehat{#1}}

\newcommand{\bw}{\textstyle{\bigwedge}}
\newcommand{\bwn}[1]{\bw^{#1}}
\newcommand{\bwo}[1]{\bw_{#1}}

\newcommand{\rubw}[1]{\bw_0^{#1}}
\newcommand{\rublat}[1]{\Omega_{#1}}

\newcommand{\rubhyp}[1]{$(\mathbf{St}#1)$}

\newcommand{\fullrke}[2]{e_{#1}[#2]}
\newcommand{\rke}[1]{\fullrke{}{#1}}

\newcommand{\sr}{r}

\newcommand{\lin}[1]{\mathcal{D}(#1)}

\newcommand{\Mor}{\mathrm{Mor}}

\newcommand{\tfwo}[2]{\bw_{#1,\mathrm{tf}}^{#2}}

\newcommand{\zinv}[1]{\Z[1/#1]}
\newcommand{\pl}{v}
\newcommand{\Pl}{w}

\newcommand{\gap}{\vspace{0.25cm}}

\newtheorem{theorem}{Theorem}[section]
\newtheorem{prop}[theorem]{Proposition}
\newtheorem{lemma}[theorem]{Lemma}
\newtheorem{cor}[theorem]{Corollary}
\newtheorem{definition}[theorem]{Definition}
\newtheorem{conj}[theorem]{Conjecture}

\numberwithin{equation}{section}
\numberwithin{table}{section}

\newenvironment{remark}{\vspace{0.25cm} \noindent \textbf{Remark}.}{\vspace{0.25cm}}

\newenvironment{proof}{\emph{Proof}.}{\hspace{\stretch{1}}\rule{1.5ex}{1.5ex} \vspace{5 mm}}

\begin{document}

\begin{center}
\LARGE The fractional Galois ideal \\
for arbitrary order of vanishing \\
\vspace{0.25cm}
\large \textsc{Paul Buckingham} \\
\footnotesize{617 CAB, Math and Statistical Sciences, University of Alberta, Edmonton AB T6G 2G1, Canada} \\
\footnotesize{p.r.buckingham@ualberta.ca}
\end{center}

\begin{abstract}
\noindent We propose a candidate, which we call the fractional Galois ideal after Snaith's fractional ideal, for replacing the classical Stickelberger ideal associated to an abelian extension of number fields. The Stickelberger ideal can be seen as gathering information about those $L$-functions of the extension which are non-zero at the special point $s = 0$, and was conjectured by Brumer to give annihilators of class-groups viewed as Galois modules. An earlier version of the fractional Galois ideal extended the Stickelberger ideal to include $L$-functions with a simple zero at $s = 0$, and was shown by the present author to provide class-group annihilators not existing in the Stickelberger ideal. The version presented in this article deals with $L$-functions of arbitrary order of vanishing at $s = 0$, and we give evidence using results of Popescu and Rubin that it is closely related to the Fitting ideal of the class-group, a canonical ideal of annihilators.

Finally, we prove an equality involving Stark elements and class-groups originally due to B\"uy\"ukboduk, but under a slightly different assumption, the advantage being that we need none of the Kolyvagin system machinery used in the original proof.

\vspace{0.25cm}


\vspace{0.25cm}

\end{abstract}

\section{Introduction} \label{intro}

In \cite{snaith:stark}, Snaith defined a \emph{fractional ideal} $\mathcal{J}_k(L/K)$ in $\Z[G]$ associated to an abelian extension $L/K$ of number fields for each negative integer $k$, where $G = \gal{L/K}$. The motivation for doing so was to address the following problem: The Coates--Sinnott Conjecture \cite{cs:stickel} predicts that, for a chosen negative integer $k$, the classical $k$th Stickelberger ideal in $\Z[G]$ provides annihilators for the $K$-group $K_{-2k}(\oh_L)$, where $\oh_L$ is the ring of integers in $L$. However, the Stickelberger ideal is often zero, and even in the case $K = \Q$ provides non-trivial annihilators for only one of the two eigenspaces for complex conjugation on the $K$-group. Snaith's solution was to use the leading coefficients of $L$-functions at the special point $s = k$ in the definition of $\mathcal{J}_k(L/K)$, rather than values as in the case of the Stickelberger ideal. Thus $\mathcal{J}_k(L/K)$ is always non-trivial, and further, he showed that in the case $K = \Q$, it does indeed give non-trivial annihilators in both the plus and minus eigenspaces.

The present author considered in \cite{buckingham:frac} the case of $L$-function derivatives at $s = 0$ rather than at strictly negative integers, relating to the problem of ideal class-groups -- zeroth $K$-groups -- instead of higher $K$-groups. The Coates--Sinnott Conjecture discussed above was inspired by a conjecture of Brumer predicting that the zeroth Stickelberger ideal gives annihilators for the class-group, although the Stickelberger ideal is often zero in this case as well. The aim of \cite{buckingham:frac} was to provide evidence that a fractional ideal similar to Snaith's would provide more annihilators of the class-group than the Stickelberger ideal. The emphasis was on the situation where the base field $K$ was $\Q$, and particularly extensions of prime-power conductor, though the behaviour of the fractional ideal $\iJ{L/K}$ was similar when $K$ was imaginary quadratic. The importance of taking leading coefficients rather than values in the cyclotomic case was made even more apparent in \cite{bs:fracfunct}. There, we took an inverse limit of the $\iJ{L_n/\Q}$ in a cyclotomic tower $L_1 \con L_2 \con L_3 \con \cdots$, and showed that the plus part, corresponding to $L$-functions with order of vanishing $1$ at $s = 0$, annihilated a limit of class-groups via the theory of cyclotomic units, while the minus part, corresponding to $L$-functions which are non-zero at $s = 0$, annihilated that limit of class-groups via the Main Conjecture of Iwasawa Theory. See \cite[Section 6.2.1]{bs:fracfunct} for details.

The three papers \cite{snaith:stark,buckingham:frac,bs:fracfunct} mentioned above suggest that the fractional ideal may be a suitable extension of the Stickelberger ideal in those cases, and thus indicate possible refinements of the Brumer and Coates--Sinnott conjectures. However, none of those papers satisfactorily addresses what should happen when the order of vanishing of the $L$-functions exceeds $1$. The aim of the current article is to tackle exactly that issue. This will require redefining the fractional ideal -- see Definition \ref{def iJ} -- to decompose it in terms of orders of vanishing of $L$-functions. The new version coincides with the old for the parts corresponding to $L$-functions having order of vanishing $0$ or $1$ at $s = 0$. In Section \ref{sec class-groups}, we support this choice of definition by using results of Popescu and Rubin to relate the fractional ideal to annihilators of class-groups, even when the $L$-functions involved vanish to order higher than $1$.

In the final section, we use the same result of Rubin mentioned above to prove an equality, originally due to B\"uy\"ukboduk, involving class-groups and the index of Stark units, but this time under a different assumption from that of B\"uy\"ukboduk. This new assumption has the advantage that the equality can be proven using a much more straightforward argument.

\section{Preliminary algebra}

\subsection{Determinants of $F[G]$-modules} \label{sec lin}

Let $R$ be any commutative ring with non-zero identity, and $M$ a finitely generated projective module over $R$. (We have in mind the situation $R = F[G]$ where $G$ is a finite abelian group and $F$ is a field whose characteristic does not divide $|G|$, in which case all modules are projective by Maschke's Theorem.) If $h \in \End_R(M)$, we arbitrarily choose a finitely generated $R$-module $N$ such that $M \oplus N$ is free, and extend $h$ to $h \oplus 1 : M \oplus N \ra M \oplus N$. Then we define $\Det_R(h)$ to be $\Det_R(h \oplus 1)$. This is independent of the choice of $N$. We note that, as usual, $\Det_R(h_1 \of h_2) = \Det_R(h_1)\Det_R(h_2)$.

Now, let $G$ be a finite abelian group and $F$ a subfield of $\C$. The character group $\Hom(G,\C\st)$ will be denoted $\chgrp{G}$.

Given a finitely generated $F[G]$-module $M$ and a non-negative integer $\sr$, let $\fullrke{M}{\sr}$ denote the sum in $\C[G]$ of the idempotents corresponding to characters $\chi \in \chgrp{G}$ for which $\blf{\chi}{\chi_M} = \sr$, where $\chi_M$ is the character of the representation $M \teno{F} \C$ and $\blf{\cdot}{\cdot}$ is the usual Hermitian product on the space of class functions on $G$. In fact, $\fullrke{M}{\sr} \in F[G]$.

\begin{definition}
Let $M$ be a finitely generated $F[G]$-module. We define the determinant of $M$, denoted $\lin{M}$, to be the $F[G]$-module
\[ \bigoplus_{\sr=0}^\infty \fullrke{M}{\sr} \bwo{F[G]}^\sr M .\]
\end{definition}

Note that, since $\chgrp{G}$ is finite, $\fullrke{M}{\sr} = 0$ for $\sr$ large enough.

\begin{prop}
For any finitely generated $F[G]$-module $M$, $\lin{M}$ is a free $F[G]$-module on one generator.
\end{prop}

\begin{proof}
This follows from the readily proven fact that for every $\sr \geq 0$, $\fullrke{M}{\sr} \bwo{F[G]}^\sr M \cong \fullrke{M}{\sr} F[G]$ (non-canonically).
\end{proof}

Suppose $M_1$ and $M_2$ are finitely generated $F[G]$-modules which are isomorphic to each other. Then given $\alpha \in \Hom_{F[G]}(M_1,M_2)$, we can form the homomorphism $\lin{\alpha} \in \Hom_{F[G]}(\lin{M_1},\lin{M_2})$ by taking exterior powers and applying the idempotents  $\fullrke{M_1}{\sr} = \fullrke{M_2}{\sr}$. In this way, $\lin{-}$ can be thought of as a functor from the category whose objects are finitely generated $F[G]$-modules with morphisms
\[ \Mor(M_1,M_2) = \left\{
\begin{array}{ll}
\Hom_{F[G]}(M_1,M_2) & \textrm{if $M_1 \cong M_2$ over $F[G]$} \\
\emptyset & \textrm{otherwise} ,
\end{array}
\right. \]
to the category of rank $1$ free $F[G]$-modules.

In terms of taking determinants of endomorphisms of finitely generated $F[G]$-modules, $\lin{-}$ can be viewed as an appropriate analogue of the maximal exterior power of a finite dimensional vector space. Indeed,

\begin{prop} \label{det of lin}
If $M$ is a finitely generated $F[G]$-module and $\alpha \in \End_{F[G]}(M)$, then $\Det_{F[G]}(\lin{\alpha}) = \Det_{F[G]}(\alpha)$.
\end{prop}

\begin{proof}
A simple argument allows one to reduce to the case where $M$ is free, and then the result is standard.
\end{proof}

\section{$T$-modified $L$-functions}

The following notions will be useful later in the paper.

\begin{definition} \label{def basic triple}
A basic triple is a triple $(L/K,S,T)$ where $L/K$ is an abelian extension of global fields, $S$ is a finite non-empty set of places of $K$, containing the infinite ones if $\chr(K) = 0$, and $T$ is a finite non-empty set of places of $K$ disjoint from $S$.
\end{definition}

Take a basic triple $(L/K,S,T)$ and let $G = \gal{L/K}$. For a character $\chi$ of $G$, define the $T$-modified $L$-function $L_{L/K,S,T}(s,\chi)$ by
\[ L_{L/K,S,T}(s,\chi) = L_{L/K,S}(s,\chi) \prod_{\pl \in T} \chr_{\chi,\pl}(\absnm\pl^{1-s}) .\]
(Here, $\chr_{\chi,\pl}(x) \in \C[x]$ denotes the characteristic polynomial of the action of the Frobenius at $v$ on the inertia fixed-points of a realization of $\chi$.) We then let $\stick_{L/K,S,T} = \sum_{\chi \in \chgrp{G}} L_{L/K,S,T}^*(0,\chi) e_{\br{\chi}} \in \R[G]\st$, where for each $\chi \in \chgrp{G}$, $L_{L/K,S,T}^*(0,\chi)$ denotes the leading coefficient of the Talyor series of $L_{L/K,S,T}(s,\chi)$ at $s = 0$, and $e_\chi$ denotes the corresponding idempotent in $\C[G]$.

\section{Rubin's Conjecture} \label{sec rub conj}

In a series of papers culminating in \cite{stark:iv}, Stark predicted an intriguing connection between derivatives of $L$-functions at $s = 0$ and regulators, a connection which can be viewed as an equivariant refinement of the analytic class number formula. A more modern formulation due to Tate can be found in \cite{tate:stark}. In fact, we will be concerned with a stronger form of this conjecture, due to Rubin and outlined in Section \ref{sec state rub conj}, which predicts that the derivatives of $L$-functions can in fact be viewed as the regulators of specific elements lying in the units of the extension field, or rather, in an exterior power of the units.

Before stating Rubin's Conjecture, we need to introduce some notation. As explained in \cite[Section 1.2]{rubin:integralstark}, if $M$ is a $\Z[G]$-module ($G$ an arbitrary finite abelian group), then for each $\sr \geq 0$ there is a well-defined homomorphism
\begin{eqnarray*}
\bwn{\sr} \Hom_{\Z[G]}(M,\Z[G]) &\ra& \Hom_{\Z[G]}(\bwn{\sr} M,\Z[G]) \\
\phi_1 \wedge \cdots \wedge \phi_\sr &\mapsto& \left(m_1 \wedge \cdots \wedge m_\sr \mapsto \Det(\phi_i(m_j)) \right) .
\end{eqnarray*}

By abuse of notation, we will denote the image of $\phi_1 \wedge \cdots \wedge \phi_\sr$ under this map by the same symbol, so that given $m_1,\ldots,m_\sr \in M$ we write simply
\[ (\phi_1 \wedge \cdots \wedge \phi_\sr)(m_1 \wedge \cdots \wedge m_\sr) = \Det(\phi_i(m_j)) .\]
We will also extend the map $(\phi_1 \wedge \cdots \wedge \phi_\sr)(-)$ linearly to $(\bwn{\sr} M) \teno{\Z} \Q \ra \Q[G]$.

\begin{definition}
For any $\Z[G]$-module $M$ and any $\sr \geq 0$, define $\rubw{\sr} M$ to be
\[ \{m \in (\bwn{\sr} M) \teno{\Z} \Q \sat \textrm{$(\phi_1 \wedge \cdots \wedge \phi_\sr)(m) \in \Z[G]$ for all $\phi_i \in \Hom_{\Z[G]}(M,\Z[G])$} \} .\]
\end{definition}

If $(L/K,S,T)$ is a basic triple, we let $r(\chi) = \blf{\chi}{\oh_{L,S}\st \teno{\Z} \C}$, which is also the order of vanishing of the $L$-function $L_{L/K,S}(s,\chi)$ at $s = 0$ by \cite[Ch.I, Prop.3.4]{tate:stark}. ($r(\chi)$ depends only on $S$, not $T$.) Then set
\[ U_{S,T} = \{u \in \oh_{L,S}\st \sat \textrm{$u \equiv 1 \Mod \Pl$ for all places $\Pl$ of $L$ above $T$}\} \]
and
\[ \rublat{S,T,\sr} = \{u \in \rubw{\sr} U_{S,T} \sat \textrm{$e_\chi u = 0$ for all $\chi \in \chgrp{G}$ with $r(\chi) \not= \sr$}\} .\]
Since $U_{S,T}$ has maximal rank in $\oh_{L,S}\st$, we can identify $U_{S,T} \teno{\Z} \Q$ with $\oh_{L,S}\st \teno{\Z} \Q$, so $\rublat{S,T,\sr}$ can be naturally viewed as a $\Z[G]$-submodule of $\rke{\sr} \bwo{\Q[G]}^\sr (\oh_{L,S}\st \teno{\Z} \Q)$ and hence of $\lin{\oh_{L,S}\st \teno{\Z} \Q}$, where
\[ \rke{\sr} = \sum_{\substack{\chi \in \chgrp{G} \\ r(\chi) = \sr}} e_\chi \in \Q[G] .\]
(Note that $\rke{\sr}$ is the idempotent $\fullrke{M}{\sr}$ defined in Section \ref{sec lin} with $F = \Q$ and $M = \oh_{L,S}\st \teno{\Z} \Q$.)

\subsection{Statement of Rubin's Conjecture} \label{sec state rub conj}

Consider the regulator map $\lambda : \oh_{L,S}\st \teno{\Z} \R \ra X \teno{\Z} \R$, where $X$ is the kernel of the augmentation map on the free abelian group on the set $S_L$ of places of $L$ above those in $S$. We remind the reader that for $u \in \oh_{L,S}\st$, $\lambda(u \ten 1) = \sum_{\pp \in S_L} \log \|u\|_\pp \cdot \pp$, where the absolute values $\|\cdot\|_\pp$ are normalized so that the product formula holds.

Let $(L/K,S,T)$ be a basic triple and $\sr \geq 0$ an integer. The following hypotheses are required:

\begin{tabular}{cp{10.5cm}}
\rubhyp{1} & $S$ contains the ramified places, and the infinite ones if $\chr(K) = 0$. \\
\rubhyp{2} & $S$ contains at least $\sr$ places which split completely in $L/K$. \\
\rubhyp{3} & $S$ contains at least $\sr+1$ places. \\
\rubhyp{4} & $U_{S,T}$ is $\Z$-torsion free.
\end{tabular}

Denote by $\lambda^{(\sr)} : (\bwo{\Z[G]}^\sr \oh_{L,S}\st) \teno{\Z} \R \ra (\bwo{\Z[G]}^\sr X) \teno{\Z} \R$ the map induced by the regulator $\lambda$. Also, if $A$ is a subring of $\C$ and $M$ is an $A[G]$-module, write $\tfwo{A[G]}{\sr} M$ for the image of $\bwo{A[G]}^\sr M$ in $(\bwo{A[G]}^\sr M) \teno{A} \C$. The following is \cite[Conjecture B]{rubin:integralstark}:

\begin{conj} \label{rubin's conj}
Granted the hypotheses \rubhyp{1} to \rubhyp{4},
\[ \rke{\sr} \stick_{L/K,S,T} \tfwo{\Z[G]}{\sr} X \con \lambda^{(\sr)}(\rublat{S,T,\sr}) .\]
\end{conj}

Assuming Conjecture \ref{rubin's conj} holds, there is a unique $\Z[G]$-submodule $\E_{S,T,\sr}$ of $\rublat{S,T,\sr}$ such that $\rke{\sr} \stick_{L/K,S,T} \tfwo{\Z[G]}{\sr} X = \lambda^{(\sr)}(\E_{S,T,\sr})$, and we call $\E_{S,T,\sr}$ the group of \emph{rank $\sr$ Stark elements}. The quotient $\rublat{S,T,\sr}/\E_{S,T,\sr}$ is a finite $\Z[G]$-module which will be of significant interest to us.

\begin{remark}
$\E_{S,T,\sr}$ is generated by one element over $\Z[G]$, and if $\eps$ is a generator, then given $x \in \Z[G]$ we have $x \eps = 0$ if and only if $\rke{\sr}x = 0$. This can be deduced from \cite[Lemma 2.6]{rubin:integralstark}, which gives a generator for $\rke{\sr} \tfwo{\Z[G]}{\sr} X$.
\end{remark}

\section{The fractional Galois ideal} \label{sec iJ}

We now come to the definition of the object which we hope will generalize the Stickelberger ideal in its relationship to class-groups.

\begin{definition} \label{def iJ}
Let $(L/K,S,T)$ be any basic triple (recall Definition \ref{def basic triple}) and $G = \gal{L/K}$. Then we define the fractional Galois ideal $\iJ{L/K,S,T}$ to be the set
\[ \stick_{L/K,S,T} \Bigg\{\Det_{\R[G]}(\alpha \of \lin{\lambda}^{-1}) \spc \Bigg| \spc \parbox{6.5cm}{$\alpha \in \Hom_{\Q[G]}(\lin{\oh_{L,S}\st \teno{\Z} \Q},\lin{X \teno{\Z} \Q})$ and $\alpha(\rublat{S,T,\sr}) \con \rke{\sr} \tfwo{\Z[G]}{\sr} X$ for all $\sr \geq 0$} \Bigg\} .\]
\end{definition}

Because $\lin{\oh_{L,S}\st \teno{\Z} \Q}$ and $\lin{X \teno{\Z} \Q}$ are free of rank $1$ over $\Q[G]$, $\iJ{L/K,S,T}$ is a $\Z[G]$-submodule of $\R[G]$. Further, $\iJ{L/K,S,T}$ is finitely generated over $\Z[G]$. We remark that $\iJ{L/K,S,T} \con \Q[G]$ if and only if $L/K$ satisfies Stark's Conjecture, as formulated in \cite{tate:stark}. Since in Section \ref{sec rel to stark} we will be assuming Rubin's Conjecture, which is stronger than that of Stark, $\iJ{L/K,S,T}$ will lie in $\Q[G]$.

\begin{prop} \label{iJ and rke}
For any $\sr \geq 0$, $\rke{\sr} \iJ{L/K,S,T} \con \iJ{L/K,S,T}$.
\end{prop}

\begin{proof}
Take a $\Q[G]$-module homomorphism $\alpha : \lin{\oh_{L,S}\st \teno{\Z} \Q} \ra \lin{X \teno{\Z} \Q}$ such that $\alpha(\rublat{S,T,t}) \con \rke{t} \tfwo{\Z[G]}{t} X$ for all $t \geq 0$. Then defining $\tilde{\alpha} : \lin{\oh_{L,S}\st \teno{\Z} \Q} \ra \lin{X \teno{\Z} \Q}$ by
\[ \tilde{\alpha}|_{\rke{t}} = \left\{
\begin{array}{ll}
\alpha|_{\rke{t}} & \textrm{if $t = \sr$} \\
0 & \textrm{otherwise} ,
\end{array} \right.
\]
one finds that $\tilde{\alpha}$ again satisfies the integrality condition in the definition of $\iJ{L/K,S,T}$, and
\[ \rke{\sr} \stick_{L/K,S,T} \Det_{\R[G]}(\alpha \of \lin{\lambda}^{-1}) = \stick_{L/K,S,T} \Det_{\R[G]}(\tilde{\alpha} \of \lin{\lambda}^{-1}) .\]
\end{proof}

As a consequence of Proposition \ref{iJ and rke}, $\iJ{L/K,S,T}$ decomposes as
\[ \iJ{L/K,S,T} = \bigoplus_\sr \rke{\sr} \iJ{L/K,S,T} .\]
Except for replacing $\oh_{L,S}\st$ by $U_{S,T}$, $\rke{0} \iJ{L/K,S,T}$ and $\rke{1} \iJ{L/K,S,T}$ coincide with $\rke{0} \iJ{L/K,S}$ and $\rke{1} \iJ{L/K,S}$ respectively, where $\iJ{L/K,S}$ is the earlier version of the fractional ideal defined in \cite{buckingham:frac}. Further, as explained in \cite[Prop. 2.8]{buckingham:frac}, $\rke{0} \iJ{L/K,S}$ is essentially the Stickelberger ideal, being generated over $\Z[G]$ by the Stickelberger element.

\subsection{Relationship to Stark elements} \label{sec rel to stark}

We now describe the relationship of $\iJ{L/K,S,T}$ to the (conjectural) Stark elements. Known cases of Rubin's Conjecture, and therefore cases in which the following theorem becomes conditionless, will be discussed in Section \ref{known cases}. We emphasize that the following would not be true for the earlier version of the fractional ideal in \cite{buckingham:frac}.

\begin{theorem} \label{iJ and stark}
Let $(L/K,S,T)$ be a basic triple and $\sr \geq 0$ an integer, and suppose that the hypotheses \rubhyp{1} to \rubhyp{4} are satisfied. If Rubin's Conjecture holds for this data, then
\begin{equation} \label{iJ and stark eq}
\rke{\sr} \iJ{L/K,S,T} = \rke{\sr} \ann_{\Z[G]}(\rublat{S,T,\sr}/\E_{S,T,\sr}) ,
\end{equation}
where $G = \gal{L/K}$.
\end{theorem}

\begin{proof}
We begin with the observation that if $\phi : M \ra N$ and $\psi : N \ra M$ are $\R[G]$-module homomorphisms, of which at least one is an isomorphism, then $\Det_{\R[G]}(\psi \of \phi) = \Det_{\R[G]}(\phi \of \psi)$. So, take $\alpha \in \Hom_{\Q[G]}(\lin{\oh_{L,S}\st \teno{\Z} \Q},\lin{X \teno{\Z} \Q})$ such that $\alpha(\rublat{S,T,t}) \con \rke{t} \tfwo{\Z[G]}{t} X$ for all non-negative integers $t$. For the purposes of this proof, we shorten $\stick_{L/K,S,T}$ to $\stick$, and also $\rke{\sr} \stick$ to $\stick_\sr$. Then given $u \in \rublat{S,T,\sr}$, $\stick_\sr \Det_{\R[G]}(\alpha \of \lin{\lambda}^{-1}) u = \stick_\sr \Det_{\R[G]}(\lin{\lambda}^{-1} \of \alpha) u$, and this is just $\stick_\sr \lin{\lambda}^{-1} \of \alpha(u)$ because $\lin{\oh_{L,S}\st \teno{\Z} \R}$ is free of rank $1$ over $\R[G]$. However, by the choice of $\alpha$, this lies in
\begin{eqnarray*}
\stick_\sr \lin{\lambda}^{-1}(\rke{\sr} \tfwo{\Z[G]}{\sr} X) &=& \lin{\lambda}^{-1}(\stick_\sr \tfwo{\Z[G]}{\sr} X) \\
&=& \E_{S,T,\sr} .
\end{eqnarray*}
Thus
\begin{equation} \label{eq proof of j to stark}
\stick_\sr \Det_{\R[G]}(\alpha \of \lin{\lambda}^{-1}) \in \{\rke{\sr} x \sat \textrm{$x \in \R[G]$, $x \rublat{S,T,\sr} \con \E_{S,T,\sr}$}\} .
\end{equation}
Now, let $\eps$ be a $\Z[G]$-generator for $\E_{S,T,\sr}$. If $x \in \R[G]$ satisfies $x \rublat{S,T,\sr} \con \E_{S,T,\sr}$, then in particular $x \eps = y \eps$ for some $y \in \Z[G]$, and hence $\rke{\sr} x = \rke{\sr} y$ by the remark following Conjecture \ref{rubin's conj}. Therefore the set in (\ref{eq proof of j to stark}) is equal to $\{\rke{\sr} y \sat \textrm{$y \in \Z[G]$, $y \rublat{S,T,\sr} \con \E_{S,T,\sr}$}\} = \rke{\sr} \ann_{\Z[G]}(\rublat{S,T,\sr}/\E_{S,T,\sr})$. This shows the inclusion ``$\con$''.

For the converse, note that Rubin's Conjecture holding for $\sr$ implies (in particular) that the map $[\stick^{-1}] \of \lambda^{(\sr)} : \rke{\sr} \bwo{\Q[G]}^\sr (\oh_{L,S}\st \teno{\Z} \Q) \ra \rke{\sr} \bwo{\R[G]}^\sr (X \teno{\Z} \R)$ of $\Q[G]$-modules, where $[-]$ denotes multiplication by the given element, in fact defines an isomorphism onto $\rke{\sr} \bwo{\Q[G]}^\sr (X \teno{\Z} \Q)$. Then given $y \in \ann_{\Z[G]}(\rublat{S,T,\sr}/\E_{S,T,\sr})$, consider the map $[y \stick^{-1}] \of \lambda^{(\sr)} : \rke{\sr} \bwo{\Q[G]}^\sr (\oh_{L,S}\st \teno{\Z} \Q) \ra \rke{\sr} \bwo{\Q[G]}^\sr (X \teno{\Z} \Q)$, and extend it arbitrarily to a map $\alpha : \lin{\oh_{L,S}\st \teno{\Z} \Q} \ra \lin{X \teno{\Z} \Q}$. If $u \in \rublat{S,T,\sr}$, then
\begin{eqnarray*}
\alpha(u) &=& \stick^{-1} \lambda^{(\sr)}(yu) \\
&\in& \stick^{-1} \lambda^{(\sr)}(\E_{S,T,\sr}) \\
&=& \rke{\sr} \tfwo{\Z[G]}{\sr} X .
\end{eqnarray*}
Also, noting that $\lin{\lambda}$ restricted to $\rke{\sr} \bwo{\R[G]}^\sr (\oh_{L,S}\st \teno{\Z} \Q)$ defines the same map as $\lambda^{(\sr)}$,
\begin{eqnarray*}
\stick_\sr \Det_{\R[G]}(\alpha \of \lin{\lambda}^{-1}) &=& \stick_\sr \Det_{\R[G]}([y \stick^{-1}] \of \lin{\lambda} \of \lin{\lambda}^{-1}) \\
&=& \rke{\sr} \Det_{\R[G]}([y]) \\
&=& \rke{\sr} y .
\end{eqnarray*}
This shows the inclusion ``$\supseteq$'', completing the proof.
\end{proof}

\subsection{Known cases of Rubin's Conjecture} \label{known cases}

Rubin's Conjecture holds in the following cases, and in these cases Theorem \ref{iJ and stark} will therefore hold unconditionally:

\begin{tabular}{cp{10cm}}
(i) & $\chr(K) > 0$; \\
(ii) & All extensions $L/K$ such that $L/\Q$ is abelian; \\
(iii) & $K$ contains an imaginary quadratic field $k$ of class number one such that $L/k$ is abelian and $\fd{L}{K}$ is odd and divisible only by primes which split in $k/\Q$; \\
(iv) & All abelian extensions $L/K$ where $K$ is imaginary quadratic (arbitrary), in the case $\sr = 1$; \\
(v) & Arbitrary quadratic extensions $L/K$; \\
(vi) & A large class of multiquadratic extensions, in the case $\sr = 1$.
\end{tabular}

\vspace{0.25cm}

In \cite{burns:congruences}, Burns shows that Rubin's Conjecture for the extension $L/K$ follows from Burns and Flach's Equivariant Tamagawa Number Conjecture (ETNC) for the pair $(h^0(\spec(L)),\Z[\gal{L/K}])$. The original formulation of the ETNC, namely \cite[Conj.4(iv)]{bf:tamagawa}, is very technical, but in the case of the pair $(h^0(\spec(L)),\Z[\gal{L/K}])$, a more explicit equivalent conjecture is Conjecture $C(L/K)$ of \cite{burns:congruences}. We remark that Conjecture $C(L/K)$ has the following base change property for abelian extensions $L/K$ and $L/K'$ with $K \con K'$, which is \cite[Prop.4.1(a)]{bf:tamagawa}:
\begin{equation}
\textrm{If $C(L/K)$ holds, then $C(L/K')$ holds.}
\end{equation}

Case (i) in the above list was done in part by Popescu -- see \cite{popescu:rubin} -- and completed for arbitrary global function fields by Burns in \cite{burns:ffcongruences}. Case (ii) follows from Burns and Greither's proof in \cite{bg:equiv} of the ETNC for the pair $(h^0(\spec(L)),\Z[\gal{L/\Q}])$ when $L$ is cyclotomic, together with the aforementioned base-change property. Case (iii) is due to Bley's work in \cite{bley:quadim} on the ETNC for imaginary quadratic fields of class-number one (see \cite{bley:quadim} for a precise description of the known cases). (iv) can be found in \cite{stark:iv}, where the rank one version of the conjecture, which Rubin's Conjecture was based on, was first stated. (v) is proven in \cite{rubin:integralstark} itself. For details of (vi), see \cite{dst:multiquad}.

\section{The fractional Galois ideal and class-groups} \label{sec class-groups}

We now explain how to describe a link between $\iJ{L/K,S,T}$ and class-groups. Here, $\class(L)_{S,T}$ will denote the $S_L$-ray class-group modulo $T_L$, namely
\[ \class(L)_{S,T} = \frac{\{\textrm{fractional ideals of $\oh_{L,S}$ prime to $T_L$}\}}{\{f \oh_{L,S} \sat f \equiv 1 \Mod w, \forall w \in T_L\}} .\]

\begin{prop} \label{fitt}
Let $(L/K,S,T)$ be a basic triple and $\sr \geq 0$ an integer, and suppose that hypotheses \rubhyp{1} to \rubhyp{4} are met. If Conjecture \ref{rubin's conj} holds for the triple $(L/K,S',T)$ and the integer $\sr$ for all sets $S'$ satisfying \rubhyp{1} to \rubhyp{3}, then
\begin{eqnarray*}
\zinv{g} \rke{\sr} \iJ{L/K,S,T} &=& \zinv{g} \rke{\sr} \fitt_{\Z[G]}(\class(L)_{S,T}) 
\end{eqnarray*}
where $g = \fd{L}{K}$.
\end{prop}

In Proposition \ref{fitt}, $\fitt_{\Z[G]}(\class(L)_{S,T})$ is the Fitting ideal of $\class(L)_{S,T}$. More generally, if $R$ is a commutative ring and $M$ a finitely presentable $R$-module, $\fitt_R(M)$ is a canonical ideal of annihilators of $M$. Further, if $M$ can be generated by a set consisting of $n$ elements, then
\[ \ann_R(M)^n \con \fitt_R(M) \con \ann_R(M) .\]
The reader wishing to know more about Fitting ideals may consult \cite{northcott:ffr}.

Before proving Proposition \ref{fitt}, we give a lemma.

\begin{lemma} \label{fitt lemma}
Suppose $R$ is a commutative ring and $e_1,\ldots,e_n$ are mutually orthogonal idempotents in $R$ whose sum is $1$ and such that $Re_i$ is a Dedekind domain for all $i$. Let $M$ be an $R$-module such that $e_i M$ is $Re_i$-torsion free for all $i$, and let $N$ be a cyclic $R$-submodule of $M$. Finally, let $e$ be the sum of the idempotents $e_i$ such that $e_i M \not= 0$. Then if $I$ is an ideal in $R$ such that $IM = N$, we have
\[ eI = e \cdot \ann_R(M/N) .\]
\end{lemma}

\begin{proof}
It is straightforward to reduce to the case $n = 1$ and $M \not= 0$. So, suppose $IM = N$. If $N = 0$, then the statement is clear since then $I = 0$ also and so $\ann_R(M/N) = \ann_R(M) = 0 = I$. Now suppose that $N$ is non-zero. That $I \con \ann_R(M/N)$ is immediate. For the reverse inclusion, first note that since $M$ is $R$-torsion free, it embeds naturally into $M \teno{R} F$ where $F$ is the fraction field of $R$. The same consequently holds for $N$. Therefore $M = I^{-1}N = \{r b \sat r \in I^{-1}\}$, choosing a generator $b$ of $N$. Given $s \in \ann_R(M/N)$ and $r \in I^{-1}$, $srb \in N$ and so is equal to $t b$ for some $t \in R$. But then, since $N$ is generated freely by $b \not= 0$, $sr = t \in R$. Thus any $s \in \ann_R(M/N)$ satisfies $s I^{-1} \con R$, showing that $\ann_R(M/N) \con (I^{-1})^{-1} = I$.
\end{proof}

\begin{proof}[Proposition \ref{fitt}]
If $K$ is a function field, then by \cite[Cor. 3.2.2]{popescu:refined},
\begin{equation} \label{rubin/popescu}
\zinv{g} \E_{S,T,\sr} = \zinv{g} \fitt_{\Z[G]}(\class(L)_{S,T}) \rublat{S,T,\sr} .
\end{equation}
If instead $K$ is a number field, then by our assumption concerning Conjecture \ref{rubin's conj} in this case, \cite[Cor. 5.4]{rubin:integralstark} provides (\ref{rubin/popescu}).

Using Lemma \ref{fitt lemma}, we deduce from (\ref{rubin/popescu}) that
\begin{equation} \label{fitt eq}
\zinv{g} \rke{\sr} \fitt_{\Z[G]}(\class(L)_{S,T}) = \rke{\sr} \ann_{\zinv{g}[G]}(\zinv{g} \rublat{S,T,\sr}/\zinv{g} \E_{S,T,\sr}) .
\end{equation}
To see that the ring $\zinv{g}[G]$ satisfies the hypotheses of Lemma \ref{fitt lemma}, we refer the reader to \cite[Section 1.3]{popescu:refined}, for example. The torsion-freeness hypothesis is met because, writing $\zinv{g}[G]$ as the direct sum of Dedekind domains $R_i$, all modules involved are contained in one of the modules $R_i \lin{\oh_{L,S}\st \teno{\Z} \Q}$, which is isomorphic (as an $R_i$-module) to the fraction field of $R_i$. The cyclicity is a consequence of the remark following Conjecture \ref{rubin's conj}.

Since $\rublat{S,T,\sr}$ is finitely generated over $\Z[G]$,
\[ \ann_{\zinv{g}[G]}(\zinv{g} \rublat{S,T,\sr}/\zinv{g} \E_{S,T,\sr}) = \zinv{g} \ann_{\Z[G]}(\rublat{S,T,\sr}/\E_{S,T,\sr}) .\]
Combining this with (\ref{fitt eq}), we therefore have
\begin{equation} \label{class and stark}
\zinv{g} \rke{\sr} \fitt_{\Z[G]}(\class(L)_{S,T}) = \zinv{g} \rke{\sr} \ann_{\Z[G]}(\rublat{S,T,\sr}/\E_{S,T,\sr}) ,
\end{equation}
and the right-hand side of (\ref{class and stark}) is \textrm{$\zinv{g} \rke{\sr} \iJ{L/K,S,T}$ by Theorem \ref{iJ and stark}}.
\end{proof}

\section{Stark elements and class-groups} \label{sec buy}

We now take a setup inspired by \cite{buyukboduk:kolofstark}: $K$ is a totally real field of degree $\sr$ over $\Q$, $L/K$ a non-trivial totally real cyclic extension with Galois group $G$, $p$ a prime not dividing $\fd{L}{K}$ and such that no prime of $K$ above $p$ ramifies in $L/K$, and $\chi : G \ra \br{\Q}_p\st$ a faithful character. We also view $\br{\Q}_p$ as a subfield of $\C$ by fixing an isomorphism $\Cp \ra \C$. The setup of \cite{buyukboduk:kolofstark} further demands that the sets $S$ and $T$ be chosen as follows: $S$ is any finite set of places of $K$ containing the $\sr$ infinite ones and the ones which ramify in $L/K$, and assume that $S$ contains at least $\sr + 1$ places but no places above $p$ and no finite places that split completely in $L/K$. Such an $S$ can always be chosen. $T$ is any finite set of places of $K$ disjoint from $S$ and such that \rubhyp{4} is satisfied (which in our case happens whenever $T$ contains a finite prime not above $2$ since $L$ is real), and such that $p$ does not divide $\absnm \pp - 1$ for any place $\pp$ of $L$ above $T$. For example, the set $T$ consisting of the places of $K$ above $p$ will satisfy all the conditions we require.

From here on, $\class(L)$ will denote the ordinary class-group of $L$, \ie the class-group of the Dedekind domain $\oh_L$.

In the main theorem of \cite{buyukboduk:kolofstark}, B\"uy\"ukboduk proves the equality (\ref{buy equality}) appearing in Corollary \ref{cor buy different assumption} below under the assumption not that Rubin's Conjecture hold for a fixed $L/K$ and varying $S'$, but that it hold for a fixed set of places and varying extensions $K'/K$. The machinery used in \cite{buyukboduk:kolofstark} involves the theory of Kolyvagin systems. We include Theorem \ref{buy different assumption} and its corollary because of the comparative ease with which they can be proven under our assumption (varying $S'$) as opposed to that of \cite{buyukboduk:kolofstark} (varying $K'/K$). We emphasize that aside from this, all other assumptions we make are also made in \cite{buyukboduk:kolofstark}.

Let $R_\chi = \Zp[\chi]$, the valuation ring of $\Qp(\chi)$. We also let $A_\chi = e[\chi] \Zp[G]$, where $e[\chi]$ is the sum of the idempotents $e_\psi$ with $\psi$ in the orbit of $\chi$ under the action of $\abqp$. Note that $e[\chi] \in \Zp[G]$ and $A_\chi$ is isomorphic to $R_\chi$. Part (i) of Theorem \ref{buy different assumption} allows us, once we extend scalars to $R_\chi$, to deal with the character $\chi$ individually. The benefit of part (ii), on the other hand, is that by considering all characters in the orbit of $\chi$, we need only extend scalars to $\Zp$.

\begin{theorem} \label{buy different assumption}
Suppose that for all sets $S'$ satisfying \rubhyp{1} to \rubhyp{3}, Rubin's Conjecture is true for the basic triple $(L/K,S',T)$ and the integer $\sr$. This happens, for example, if $L/\Q$ is abelian. Let $\eps$ be a $\Z[G]$-generator for $\E_{S,T,\sr}$.

(i) The element $\eps^\chi = e_\chi(\eps \ten 1) \in e_\chi (\tfwo{\Z[G]}{\sr} U_{S,T}) \teno{\Z} R_\chi$ can be viewed as an element of $e_\chi \tfwo{R_\chi[G]}{\sr} (\oh_L\st \teno{\Z} R_\chi)$, and
\begin{equation} \label{fitt equality}
\fitt_{R_\chi}(e_\chi(\class(L) \teno{\Z} R_\chi)) = \fitt_{R_\chi}(e_\chi \tfwo{R_\chi[G]}{\sr} (\oh_L\st \teno{\Z} R_\chi)/R_\chi \eps^\chi) .
\end{equation}

(ii) The element $\eps^{[\chi]} = e[\chi](\eps \ten 1) \in e[\chi] (\tfwo{\Z[G]}{\sr} U_{S,T}) \teno{\Z} \Zp$ can be viewed as an element of $e[\chi] \tfwo{\Zp[G]}{\sr} (\oh_L\st \teno{\Z} \Zp)$, and
\[ \fitt_{A_\chi}(e[\chi] \class(L) \teno{\Z} \Zp) = \fitt_{A_\chi}(e[\chi] \tfwo{\Zp[G]}{\sr} (\oh_L\st \teno{\Z} \Zp)/A_\chi \eps^{[\chi]}) .\]
\end{theorem}

\begin{proof}
Let us deal with part (i). The claim about $\eps^\chi$ will be proven along the way. Now, by the choice of $T$, $\class(L)_{S,T} \teno{\Z} \Zp = \class(L)_S \teno{\Z} \Zp$, where $\class(L)_S$ is the $S_L$-class-group of $L$. Further, using \cite[Prop. 11.6]{neukirch:alg} gives that $e_\chi \class(L)_S \teno{\Z} R_\chi = e_\chi \class(L) \teno{\Z} R_\chi$, so that the left-hand side of (\ref{fitt equality}) becomes
\begin{equation} \label{buy step' 2}
\fitt_{R_\chi}(e_\chi(\class(L)_{S,T} \teno{\Z} R_\chi)) = \chi(e_\chi \fitt_{R_\chi[G]}(\class(L)_{S,T} \teno{\Z} R_\chi)) .
\end{equation}

We now appeal to (\ref{class and stark}), obtaining that the right-hand side of (\ref{buy step' 2}) is equal to
\begin{eqnarray}
& & \chi(e_\chi \ann_{R_\chi[G]}((\rublat{S,T,\sr}/\E_{S,T,\sr}) \teno{\Z} R_\chi)) \nonumber \\
&=& \ann_{R_\chi}(e_\chi(\rublat{S,T,\sr} \teno{\Z} R_\chi/\E_{S,T,\sr} \teno{\Z} R_\chi))) \label{buy step' 3} .
\end{eqnarray}
Since $\fd{L}{K} \in \Zp\st$ and 
\[ e_\chi \tfwo{\Z[G]}{\sr} U_{S,T} \con e_\chi \rublat{S,T,\sr} \con e_\chi \frac{1}{\fd{L}{K}} \tfwo{\Z[G]}{\sr} U_{S,T} ,\]
(\ref{buy step' 3}) is therefore equal to
\begin{equation} \label{buy step' 4}
\ann_{R_\chi}(e_\chi((\tfwo{\Z[G]}{\sr} U_{S,T}) \teno{\Z} R_\chi/\E_{S,T,\sr} \teno{\Z} R_\chi)) .
\end{equation}

Now, $e_\chi \big(\tfwo{\Z[G]}{\sr} U_{S,T}\big) \teno{\Z} R_\chi \can \tfwo{R_\chi[G]}{\sr} e_\chi(U_{S,T} \teno{\Z} R_\chi)$, and by the choice of $T$, $U_{S,T} \teno{\Z} R_\chi = \oh_{L,S}\st \teno{\Z} R_\chi$. Also, $e_\chi \E_{S,T,\sr} \teno{\Z} R_\chi = R_\chi \eps^\chi$ since $\sigma \in G$ acts as $\chi(\sigma) \in R_\chi\st$ on $e_\chi \E_{S,T,\sr} \teno{\Z} R_\chi$. We therefore obtain that (\ref{buy step' 4}) is equal to
\begin{equation} \label{buy step' 5}
\ann_{R_\chi}(\tfwo{R_\chi[G]}{\sr} e_\chi(\oh_{L,S}\st \teno{\Z} R_\chi)/R_\chi \eps^\chi) .
\end{equation}

Since the $S$-rank of $\chi$ and the $S_\infty$-rank of $\chi$ are both the same by \cite[Ch.I, Prop.3.4]{tate:stark}, where $S_\infty$ is the set of infinite places of $K$, $\rk_{R_\chi} e_\chi (\oh_{L,S}\st \teno{\Z} R_\chi) = \rk_{R_\chi} e_\chi(\oh_L\st \teno{\Z} R_\chi)$. Using again the exact sequence in \cite[Prop.11.6]{neukirch:alg} we find that $\oh_{L,S}\st/\oh_L\st$ is torsion-free, therefore so is $e_\chi (\oh_{L,S}\st \teno{\Z} R_\chi)/e_\chi (\oh_L\st \teno{\Z} R_\chi)$. But this quotient has $R_\chi$-rank zero by the above, and hence is trivial, \ie
\[ e_\chi(\oh_{L,S}\st \teno{\Z} R_\chi) = e_\chi(\oh_L\st \teno{\Z} R_\chi) .\]
Thus (\ref{buy step' 5}) is
\begin{equation} \label{buy step' 6}
\ann_{R_\chi}(e_\chi \tfwo{R_\chi[G]}{\sr} (\oh_L\st \teno{\Z} R_\chi)/R_\chi \eps^\chi) .
\end{equation}
Since $e_\chi \tfwo{R_\chi[G]}{\sr} (\oh_L\st \teno{\Z} R_\chi)$ is torsion-free and the quotient by $R_\chi \eps^\chi$ is finite, this quotient must be cyclic as an $R_\chi$-module, so that its annihilator ideal equals its Fitting ideal. Hence we obtain that (\ref{buy step' 6}) is equal to $\fitt_{R_\chi}(e_\chi \tfwo{R_\chi[G]}{\sr} (\oh_L\st \teno{\Z} R_\chi)/R_\chi \eps^\chi)$, as required.

Addressing part (ii), we remark that almost all of its proof is identical to that of part (i). In fact, it is even simpler because we only need to extend scalars to $\Zp$, and, aside from the details already given in the proof of part (i), it uses only the behaviour of Fitting ideals and annihilators when applying idempotents.
\end{proof}

\begin{cor} \label{cor buy different assumption}
If, in addition to the assumptions of Theorem \ref{buy different assumption}, we suppose that $p \equiv 1 \Mod \fd{L}{K}$ (as assumed in \cite{buyukboduk:kolofstark}), then $\chi$ takes its values in $\Zp$ and
\begin{equation} \label{buy equality}
|e_\chi \class(L) \teno{\Z} \Zp| = |e_\chi \tfwo{\Zp[G]}{\sr} (\oh_L\st \teno{\Z} \Zp):\Zp \eps^\chi| .
\end{equation}
\end{cor}

\begin{proof}
The assertion about $\chi$ holds because $\Zp$ contains the $(p-1)$th roots of unity. Since $e_\chi \class(L) \teno{\Z} \Zp$ and $e_\chi \tfwo{\Zp[G]}{\sr} (\oh_L\st \teno{\Z} \Zp)/\Zp \eps^\chi$ are finite $p$-groups, in order to demonstrate (\ref{buy equality}) it is equivalent to show that the ideals in $\Zp$ generated by their orders are equal, but these ideals are their Fitting ideals over $\Zp$. Now apply part (i) of Theorem \ref{buy different assumption}.
\end{proof}

\gap

\noindent \begin{large} \textbf{Acknowledgments} \end{large}

\gap

\noindent The author would like to thank the referee for reading the manuscript carefully, and for making valuable suggestions.


\begin{thebibliography}{10}

\bibitem{bley:quadim}
W.~Bley.
\newblock Equivariant {T}amagawa number conjecture for abelian extensions of a
  quadratic imaginary field.
\newblock {\em Doc. Math.}, 11:73--118 (electronic), 2006.

\bibitem{bs:fracfunct}
P.~Buckingham and V.~Snaith.
\newblock Functoriality of the canonical fractional {G}alois ideal.
\newblock {\em To appear in the Canadian Journal of Mathematics}.

\bibitem{buckingham:frac}
Paul Buckingham.
\newblock The canonical fractional {G}alois ideal at $s = 0$.
\newblock {\em J. Number Theory}, 128(6):1749--1768, 2008.

\bibitem{bf:tamagawa}
D.~Burns and M.~Flach.
\newblock Tamagawa numbers for motives with (non-commutative) coefficients.
\newblock {\em Doc. Math.}, 6:501--570 (electronic), 2001.

\bibitem{burns:congruences}
David Burns.
\newblock Congruences between derivatives of abelian {$L$}-functions at
  {$s=0$}.
\newblock {\em Invent. Math.}, 169(3):451--499, 2007.

\bibitem{burns:ffcongruences}
David Burns.
\newblock Congruences between derivatives of geometric {$L$}-functions.
\newblock {\em Preprint}, 2008.

\bibitem{bg:equiv}
David Burns and Cornelius Greither.
\newblock On the equivariant {T}amagawa number conjecture for {T}ate motives.
\newblock {\em Invent. Math.}, 153(2):303--359, 2003.

\bibitem{buyukboduk:kolofstark}
K{\^a}zim B{\"u}y{\"u}kboduk.
\newblock Kolyvagin systems of {S}tark units.
\newblock {\em J. Reine Angew. Math.}, 631:85--107, 2009.

\bibitem{cs:stickel}
John Coates and Warren Sinnott.
\newblock An analogue of {S}tickelberger's theorem for the higher {$K$}-groups.
\newblock {\em Invent. Math.}, 24:149--161, 1974.

\bibitem{dst:multiquad}
David~S. Dummit, Jonathan~W. Sands, and Brett Tangedal.
\newblock Stark's conjecture in multi-quadratic extensions, revisited.
\newblock {\em J. Th\'eor. Nombres Bordeaux}, 15(1):83--97, 2003.
\newblock Les XXII\`emes Journ\'ees Arithmetiques (Lille, 2001).

\bibitem{neukirch:alg}
J{\"u}rgen Neukirch.
\newblock {\em Algebraic number theory}, volume 322 of {\em Grundlehren der
  Mathematischen Wissenschaften [Fundamental Principles of Mathematical
  Sciences]}.
\newblock Springer-Verlag, Berlin, 1999.
\newblock Translated from the 1992 German original and with a note by Norbert
  Schappacher, With a foreword by G. Harder.

\bibitem{northcott:ffr}
D.~G. Northcott.
\newblock {\em Finite free resolutions}.
\newblock Cambridge University Press, Cambridge, 1976.
\newblock Cambridge Tracts in Mathematics, No. 71.

\bibitem{popescu:refined}
Cristian~D. Popescu.
\newblock On a refined {S}tark conjecture for function fields.
\newblock {\em Compositio Math.}, 116(3):321--367, 1999.

\bibitem{popescu:rubin}
Cristian~D. Popescu.
\newblock Rubin's integral refinement of the abelian {S}tark conjecture.
\newblock In {\em Stark's conjectures: recent work and new directions}, volume
  358 of {\em Contemp. Math.}, pages 1--35. Amer. Math. Soc., Providence, RI,
  2004.

\bibitem{rubin:integralstark}
Karl Rubin.
\newblock A {S}tark conjecture ``over {$\bold Z$}'' for abelian {$L$}-functions
  with multiple zeros.
\newblock {\em Ann. Inst. Fourier (Grenoble)}, 46(1):33--62, 1996.

\bibitem{snaith:stark}
Victor~P. Snaith.
\newblock Stark's conjecture and new {S}tickelberger phenomena.
\newblock {\em Canad. J. Math.}, 58(2):419--448, 2006.

\bibitem{stark:iv}
Harold~M. Stark.
\newblock {$L$}-functions at {$s=1$}. {IV}. {F}irst derivatives at {$s=0$}.
\newblock {\em Adv. in Math.}, 35(3):197--235, 1980.

\bibitem{tate:stark}
John Tate.
\newblock {\em Les conjectures de {S}tark sur les fonctions {$L$} d'{A}rtin en
  {$s=0$}}, volume~47 of {\em Progress in Mathematics}.
\newblock Birkh\"auser Boston Inc., Boston, MA, 1984.
\newblock Lecture notes edited by Dominique Bernardi and Norbert Schappacher.

\end{thebibliography}
\end{document}